\documentclass{amsart}
\usepackage{times}
\usepackage{amsthm}
\numberwithin{equation}{section}
\usepackage[]{graphicx}

\newtheorem{theorem}{Theorem}[section]
\newtheorem{lemma}[theorem]{Lemma}
\newtheorem{prop}[theorem]{Proposition}
\theoremstyle{definition}
\newtheorem{definition}[theorem]{Definition}

\theoremstyle{remark}
\newtheorem{remark}[theorem]{Remark}

\newcommand{\Ext}{\operatorname{Ext}}

\newcommand{\Dim}{\operatorname{dim}}

\newcommand{\Pic}{\operatorname{Pic}}

\newcommand{\Ker}{\operatorname{Ker}}

\newcommand{\Coker}{\operatorname{Coker}}
\newcommand{\Hom}{\operatorname{Hom}}
\newcommand{\rango}{\operatorname{rk}}
\newcommand{\p}{{\mathbb{P}^1}}
\newcommand{\pp}{{\mathbb{P}}}
\newcommand{\ppp}{{\mathbb{P}^1 \times \mathbb{P}^1 \times \mathbb{P}^1 }}

\newcommand{\sO}{\mathcal{O}}

\newcommand\sD{{\mathcal D}}

\newcommand\sH{{\mathcal H}}
\newcommand\sI{{\mathcal I}}

\newcommand\sL{{\mathcal L}}

\newcommand\sN{{\mathcal N}}
  {\left\lbrace\begin{array}{@{}l@{}}}%
  {\end{array}\right.}

\usepackage{array}
\usepackage{tabu}
\usepackage{amssymb}
\usepackage {amsmath}
\usepackage {amsthm}
\usepackage {amscd}
\usepackage {amssymb}
\usepackage {latexsym}
\usepackage{multirow}
\usepackage{tikz-cd}
\usepackage{xypic}

\usepackage[a4paper,left=3.3cm,right=3.3cm]{geometry}
\usepackage{verbatim}
\usepackage{mathrsfs}
\usepackage{amsfonts}
 
\begin{document}

\title[Instanton bundles on the Segre threefold with Picard Number three]{Instanton bundles on the Segre threefold\\ with Picard Number three
}

\author{V. Antonelli, F. Malaspina }

\address{Politecnico di Torino, Corso Duca degli Abruzzi 24, 10129 Torino, Italy}
\email{vincenzo.antonelli@polito.it}

\address{Politecnico di Torino, Corso Duca degli Abruzzi 24, 10129 Torino, Italy}
\email{francesco.malaspina@polito.it}


\begin{abstract}
We study instanton bundles $E$ on $\ppp$. We construct two different monads which are the analog of the monads for instanton bundles on $\mathbb P^3$ and on the flag threefold $F(0,1,2)$. We characterize the Gieseker semistable cases and we prove the existence of $\mu$-stable instanton bundles generically trivial on the lines for any possible $c_2(E)$. We also study the locus of jumping lines.

\end{abstract}
\maketitle                   






\section{Introduction}

Instanton bundles on $\mathbb P^3$ were first defined in \cite{AHDM} by Atiyah, Drinfel'd, Hitchin and Manin. Their importance arises from quantum physics; in fact these particular bundles correspond (through the Penrose-Ward transform) to self dual solutions of the Yang-Mills equation over the real sphere  $S^4$.
   We recall that a mathematical instanton bundle $E$ with charge (or quantum number $k$) on $\mathbb P^3$ is a stable rank two vector bundle $E$ with $c_1(E)=0$, $c_2(E)=k$ and with the property (called instatonic condition) that
$$H^1(E(-2))=0.$$
Every instanton of charge $k$ on $\mathbb P^3$  can be represented as the cohomology of a monad (a three-term self dual complex).

In \cite{Hit}, Hitchin showed that the only twistor spaces of four dimensional (real) differentiable manifolds which are K\"{a}hler (and a posteriori, projective) are $\mathbb P^3$ and
the flag variety $F(0,1,2)$, which is the twistor space of $\mathbb P^2$.

On $F(0,1,2)$ instanton bundles has been studied in \cite{Bu}, \cite{Don} and more recently in \cite{MMP}. $F(0,1,2)$ is a Fano threefold with Picard number two. Let us call $h_1$ and $h_2$ the two generators. In \cite{MMP} has been given the following definition: a rank  two vector bundle $E$ on the Fano threefold $F(0,1,2)$ is an instanton bundle of charge $k$ if the following properties hold
\begin{itemize}
\item $c_1(E)=0, c_2(E)=kh_1h_2$;
\item $h^0(E)=h^1(E(-1,-1))=0$ and $E$ is $\mu$-semistable;
\end{itemize}
Notice that, when the Picard number is one, the condition $H^0(E)=0$ implies the $\mu$-stability. When the Picard number is higher than one, however, this is not true and it is natural to consider also $\mu$-semistable bundles (see \cite{Don} and \cite{MMP} Remark 2.2).

In \cite{Fa2} (see also \cite{Kuz} in the case $i_F=2$  and \cite{Sa} for details in the case of the del Pezzo threefold of degree $5$), the author generalizes the notion of instanton bundle on $\mathbb P^3$ to any Fano threefold  with Picard number one. In this line may be generalized also the definition on $F(0,1,2)$ to any Fano threefold  with Picard number higher than one (see \cite{CCGM} for the case of  the blow up of the projective $3$-space at a point).

The subspace of stable instanton bundles with a given $c_2$ can be identified with the open subspace of the Maruyama moduli space  of stable rank two bundles with those fixed Chern classes satisfying the cohomological vanishing condition. For  a large family of Fano threefolds with Picard number one, Faenzi  in \cite[Theorem A]{Fa2} proves that the moduli spaces of instanton bundles has a generically smooth irreducible component. An  analogous result has been obtained in \cite{MMP} for the flag threefold. In the case of
$\mathbb P^3$, it is known that the moduli space of instantons of arbitrary charge is affine (see \cite{CO}), irreducible (see \cite{T1}, \cite{T2}) and smooth (see \cite{CTT}, \cite{KaOtt} for charge smaller than $5$ and \cite{JaVe} for arbitrary charge). The rationality is still an open problem in general, being settled only for charges $1$, $2$, $3$, and $5$ (see \cite{Ba2}, \cite{Har2}, \cite{ES} and \cite{Kat}).

In this paper we consider $\ppp$ which has the same index and degree of $F(0,1,2)$ but Picard number three. Let us call $h_1$, $h_2$ and $h_3$ the three generators. The only difference with respect to the definition of instanton bundle on $F(0,1,2)$ is that on $\ppp$ the second Chern class is $c_2(E)=k_1h_2h_3+k_2h_1h_3+k_3h_1h_2$ instead of $c_2(E)=kh_1h_2$. By using a Beilinson type spectral sequence with suitable full exceptional collections we construct two different monads which are the analog of the monads for instanton bundles on $\mathbb P^3$ and on $F(0,1,2)$.

\begin{theorem}
Let $E$ be a charge $k$ instanton bundle on $X$ with $c_2(E)=k_1e_1+k_2e_2+k_3e_3$, then $E$ is the cohomology of a monad of the form
\begin{itemize}
\item [(i)]
\begin{equation*}
0\rightarrow
\begin{matrix}
\sO_X^{k_3}(-h_1-h_2) \\
\oplus \\
\sO_X^{k_2}(-h_1-h_3) \\
\oplus \\
\sO_X^{k_1}(-h_2-h_3)
\end{matrix}
\rightarrow
\begin{matrix}
\sO_X^{k_2+k_3}(-h_1) \\
\oplus \\
\sO_X^{k_1+k_3}(-h_2) \\
\oplus\\
\sO_X^{k_1+k_2}(-h_3)
\end{matrix}
\rightarrow
\sO_X^{k-2}
\rightarrow
0.
\end{equation*}
Conversely any $\mu$-semistable bundle defined as the cohomology of such  a monad is a  charge $k$ instanton bundle.
\item[(ii)]

\begin{equation*}
0\rightarrow
\begin{matrix}
\sO_X^{k_3}(-h_1-h_2) \\
\oplus \\
\sO_X^{k_2}(-h_1-h_3) \\
\oplus \\
\sO_X^{k_1}(-h_2-h_3)
\end{matrix}
\rightarrow
\begin{matrix}
\sO_X^{3k+2}
\rightarrow
\end{matrix}
\begin{matrix}
\sO_X^{k_2+k_3}(h_1) \\
\oplus \\
\sO_X^{k_1+k_3}(h_2) \\
\oplus\\
\sO_X^{k_1+k_2}(h_3)
\end{matrix}
\rightarrow
0.
\end{equation*}
Conversely any $\mu$-semistable bundle with $H^0(E)=0$ defined as the cohomology of such  a monad is a  charge $k$ instanton bundle.

\end{itemize}
\end{theorem}
Furthermore we show that the Gieseker strictly semistable instanton bundles are extensions of line bundles and can be obtained as pullbacks from $\mathbb P^1\times\mathbb P^1$. The cases where the degree of $c_2(E)$ is minimal, namely $k=k_1+k_2+k_2=2$, has been studied in \cite{pppacm}. In fact we get, up to twist, Ulrich bundles.

Here we show that Ulrich bundles is generically trivial on the lines. So we use this case as a starting step in order to prove by induction the existence of $\mu$-stable instanton bundles generically trivial on the lines for any possible $c_2(E)$. In particular we prove the following
\begin{theorem}
For each non-negative $k_1,k_2,k_3 \in \mathbb{Z}$ with $k=k_1+k_2+k_3\geq 2$ there exists a $\mu$-stable instanton bundle $E$ with $c_2(E)=k_1e_1+k_2e_2+k_3e_3$ on $X$ such that
\[
\Ext_X^1(E,E)=4k-3, \qquad \Ext_X^2(E,E)=\Ext_X^3(E,E)=0
\]
and such that $E$ is generically trivial on lines.

In particular there exists, inside the moduli space $MI(k_1e_1+k_2e_2+k_3e_3)$ of instanton bundles with $c_2=k_1e_1+k_2e_2+k_3e_3$, a generically smooth irreducible component of dimension $4k-3$.
\end{theorem}

Finally we also study the locus of jumping lines obtaining the following result:
\begin{prop}
Let $E$ be a generic instanton on $X$ with $c_2=k_1e_1+k_2e_2+k_3e_3$. Then the locus of jumping lines in the family $|e_1|$, denoted by $\sD_E^1$, is a divisor given by $\sD_E^1=k_3l+k_2m$ equipped with a sheaf $G$ fitting into
\begin{equation*}
0\rightarrow \sO_\sH^{k_3}(-1,0)\oplus \sO_\sH^{k_2} (0,-1)\rightarrow \sO_\sH^{k_2+k_3}\rightarrow i_\ast G\rightarrow 0.
\end{equation*}
\end{prop}
Permuting indices we are also able to describe the locus of jumping lines in the other two rulings of $\ppp$.

The authors want to thank G. Casnati and J. Pons-Llopis for helpful discussions on the subject.
\section{First properties and monads of instanton bundles on $\ppp$}

Let $V_1, V_2, V_3$ be three $2$-dimensional vector spaces with the coordinates $[x_{1i}], [x_{2j}], [x_{3k}]$ respectively with $i,j,k\in \{1,2\}$. Let $X\cong \mathbb P (V_1) \times \mathbb P (V_2) \times \mathbb P (V_3)$ and then it is embedded into $\mathbb P^7\cong \mathbb P(V)$ by the Segre map where $V=V_1 \otimes V_2 \otimes V_3$.

The intersection ring $A(X)$ is isomorphic to $A(\mathbb P^1) \otimes A(\mathbb P^1) \otimes A(\mathbb P^1)$ and so we have
$$A(X) \cong \mathbb Z[h_1, h_2, h_3]/(h_1^2, h_2^2, h_3^2).$$
We may identify $A^1(X)\cong \mathbb Z^{\oplus 3}$ by $a_1h_1+a_2h_2+a_3h_3 \mapsto (a_1, a_2, a_3)$. Similarly we have $A^2(X) \cong \mathbb Z^{\oplus 3}$ by $k_1e_1+k_2e_2+k_3e_3\mapsto (k_1, k_2, k_3)$ where $e_1=h_2h_3, e_2=h_1h_3, e_3=h_1h_2$ and $A^3(X) \cong \mathbb Z$ by $ch_1h_2h_3 \mapsto c$.
Then $X$ is embedded into $\mathbb P^7$ by the complete linear system $h=h_1+h_2+h_3$ as a subvariety of degree $6$ since $h^3=6$.
  
If $E$ is a rank two bundle with the Chern classes $c_1=(a_1, a_2, a_3)$, $c_2=(k_1, k_2, k_3)$  we have:

\begin{gather}
c_1(E (s_1, s_2, s_3))=(a_1+2s_1, a_2+2s_2, a_3+2s_3) \label{chernppp}\\
c_2(E(s_1, s_2, s_3)) =c_2+c_1\cdot(s_1, s_2, s_3)+(s_1, s_2, s_3)^2 \notag
\end{gather}

for $(s_1, s_2, s_3)\in \mathbb Z^{\oplus 3}$.

Let us recall the Riemann-Roch formula:
\begin{equation}\label{RR}
\chi (E)=(a_1+1)(a_2+1)(a_3+1)+1
-\frac{1}{2}((a_1, a_2, a_3)\cdot (k_1, k_2, k_3) +2(k_1+k_2+k_3))
\end{equation}

Recall that for each torsion free sheaf $F$ on $X$ the {\sl slope} of $F$ with respect to $h$ is the rational number $\mu(F):=c_1(F)h^2/\rango(F)$ and the reduced Hilbert
polynomial $P_{E}(t)$ of a bundle $E$ over $X$ is
$ P_{E}(t):=\chi(E(th))/\rango(E).$

We say that a vector bundle $E$ is {\sl $\mu$-stable} (resp. {\sl $\mu$-semistable}) with respect to $h$ if  $\mu( G) < \mu(E)$ (resp. $\mu(G) \le \mu(E)$) for each subsheaf $ G$ with $0<\rango(G)<\rango(E)$.

On the other hand, $E$ is said to be Gieseker semistable  with respect to $h$ if for all $ G$ as above one has
$$
P_{ G}(t) \le  P_{E}(t),
$$
and Gieseker stable again if equality cannot hold in the above inequality.

\begin{definition}
A $\mu$-semistable vector bundle $E$ on $\ppp$ is called an instanton bundle of charge $k$ if and only if $c_1(E)=0$,
\begin{equation*}H^0(E)=H^1(E(-h))=0
\end{equation*}
and $c_2(E)=k_1e_1+k_2e_2+k_3e_3$ with $k_1+k_2+k_3=k$.
\end{definition}
Let us briefly recall the definition of aCM sheaves and Ulrich sheaves. Let $X\subset \pp^N$ be a projective variety which is naturally endowed with the very ample line bundle $\sO_X(h)=\sO_X \otimes \sO_{\pp^N}(1)$. We say that $X$ is arithmetically Cohen-Macaulay (aCM for short) if $H^i(\mathcal{I}_X(th))=0$ for $t\in\mathbb{Z}$ and $1\leq i \leq \Dim(X)$. Observe that the variety $\ppp \subset \pp^7$ is an aCM variety.

A coherent sheaf $E$ over an aCM variety $X$ is called aCM if all of its intermediate cohomology groups vanish, i.e. $H^i(X,E(th))=0$ for $1\leq i\leq \Dim(X)-1$. Ulrich sheaves are defined to be the aCM sheaves whose corresponding module has the maximum number of generators.
\begin{remark}\label{rem}
  It is worthwhile to point out that, exactly as in the case of $F(0,1,2)$ (see \cite{MMP} Remark 2.2), the condition $H^0(E)=0$ does not follow from the other conditions defining an instanton bundle. Indeed we may consider the rank two aCM bundles with $c_1(E)=0$ and $H^0(E)\not=0$ given in \cite{pppacm} Theorem B.
\end{remark}
Now we recall the Hoppe's criterion for semistable vector bundles over polycyclic varieties, i.e. varieties $X$ such that $\Pic(X)=\mathbb{Z}^l$.
\begin{prop}\cite[Theorem 3]{JMPS}\label{hoppe}
Let $E$ be a rank two holomorphic vector bundle over a polycyclic variety $X$ and let $L$ be a polarization on $X$. $E$ is $\mu$-(semi)stable if and only if
$$
H^0(X,E\otimes \sO_X(B))=0
$$
for all $B \in Pic(X)$ such that $\delta_L(B) \underset{(<)}{\leq} -\mu_L(E)$, where $\delta_L(B)=\deg_L(\sO_X(B))$.
\end{prop}

In order to get a monadic description of instanton bundles, we will use the following version (explained in \cite{AHMP}) of the Beilinson spectral sequence (see also  \cite[Corollary 3.3.2]{RU}, and \cite[Section 2.7.3]{GO} and  \cite[Theorem 2.1.14]{BO}).

\begin{theorem}\label{use}
Let $X$ be a smooth projective variety with a full exceptional collection
$\langle E_0, \ldots, E_n\rangle$
where $E_i=\mathcal E_i^*[-k_i]$ with each $\mathcal E_i$ a vector bundle and $(k_0, \ldots, k_n)\in \mathbb Z^{\oplus n+1}$ such that there exists a sequence $\langle F_n=\mathcal F_n, \ldots, F_0=\mathcal F_0\rangle$ of vector bundles satisfying
\begin{equation}\label{order}
\mathrm{Ext}^k(E_i,F_j)=H^{k+k_i}( \mathcal E_i\otimes \mathcal F_j) =  \left\{
\begin{array}{cc}
\mathbb C & \textrm{\quad if $i=j=k$} \\
0 & \textrm{\quad otherwise.}
\end{array}
\right.
\end{equation}

Then for any coherent sheaf $A$ on $X$ there is a spectral sequence in the square $-n\leq p\leq 0$, $0\leq q\leq n$  with the $E_1$-term
\[
E_1^{p,q} = \mathrm{Ext}^{q}(E_{-p},A) \otimes F_{-p}=
H^{q+k_{-p}}(\mathcal E_{-p}\otimes A) \otimes \mathcal F_{-p}
\]
which is functorial in $A$ and converges to
\begin{equation}
E_{\infty}^{p,q}= \left\{
\begin{array}{cc}
A & \textrm{\quad if $p+q=0$} \\
0 & \textrm{\quad otherwise.}
\end{array}
\right.
\end{equation}
\end{theorem}

Let $D^b(X)$ be the the bounded derived category of coherent sheaves on a smooth projective variety $X$.
An object $E \in D^b(X)$ is called {\emph{exceptional}} if $\Ext^\bullet(E,E) = \mathbb C$.
We recall that a  set of exceptional objects $E_1, \ldots, E_n$ on  $X$ is called an {\emph{exceptional collection}} if $\Ext^\bullet(E_i,E_j) = 0$ for $i > j$.
An exceptional collection is {\emph{full}} when $\Ext^\bullet(E_i,A) = 0$ for all $i$ implies $A = 0$, or equivalently when $\Ext^\bullet(A, E_i) = 0$ for all $i$ also implies $A = 0$.
Moreover we say that an exceptional collection is {\emph{strong}} if $\Ext^k(E_i,E_j) = 0$ for all $i$, $j$ and $k>0$.

\begin{remark}\label{rembeil}
It is possible to state a stronger version of the Beilinson's theorem (see \cite{ottanc}, \cite{Be} for $\mathbb{P}^N$ and \cite{AO3} for the projectivized of a direct sum of line bundles over $\mathbb{P}^N$). Let us consider $X=\ppp$ and let $A$ be a coherent sheaf on $X$. Let $(E_0,\dots,E_n)$ be a full exceptional collection and $(F_n,\dots,F_0)$ its right dual collection. Using the notation of Theorem \ref{use}, if $(F_n,\dots,F_0)$ is strong then there exists a complex of vector bundles $L^\bullet$ such that
\begin{enumerate}
\item $H^k(L^\bullet)=
\begin{cases}
A \ & \text{if $k=0$},\\
0 \ & \text{otherwise}.
\end{cases}$
\item $L^k=\underset{k=p+q}{\bigoplus}H^{q+k_{-p}}(A\otimes E_{-p})\otimes F_{-p}$ with $0\le q \le n$ and $-n\le p \le 0$.
\end{enumerate}
\end{remark}

\begin{definition}\label{def:mutation}
Let $E$ be an exceptional object in $D^b(X)$.
Then there are functors $\mathbb L_{E}$ and $\mathbb R_{E}$ fitting in distinguished triangles
$$
\mathbb L_{E}(T) 		\to	 \Ext^\bullet(E,T) \otimes E 	\to	 T 		 \to	 \mathbb L_{E}(T)[1]
$$
$$
\mathbb R_{E}(T)[-1]	 \to 	 T 		 \to	 \Ext^\bullet(T,E)^* \otimes E	 \to	 \mathbb R_{E}(T).	
$$
The functors $\mathbb L_{E}$ and $\mathbb R_{E}$ are called respectively the \emph{left} and \emph{right mutation functor}.
\end{definition}

Now we construct the full exceptional collections that we will use in the next theorems:
Let us consider on the three copies of $\mathbb P^1$ the full exceptional collection $\{ \sO_{\mathbb P^1}(-1), \sO_{\mathbb P^1} \}$. We may obtain the full exceptional collection $\langle E_7, \ldots, E_0\rangle$ (see \cite{Orlov}):

 \begin{gather}\label{col}
 \{ \sO_X(-h)[-4], \sO_X(-h_2-h_3)[-4], \sO_X(-h_1-h_3)[-3], \\
 \sO_X(-h_1-h_2)[-2], \sO_X(-h_3)[-2], \sO_X(-h_2)[-1], \sO_X(-h_1), \sO_X\}. \notag
\end{gather}
The associated full exceptional collection $\langle F_7=\mathcal F_7, \ldots, F_0=\mathcal F_0\rangle$ of Theorem \ref{use} is

\begin{equation}\label{cold}
\{\sO_X(-h), \sO_X(-h_2-h_3), \sO_X(-h_1-h_3), \sO_X(-h_1-h_2), \sO_X(-h_3), \sO_X(-h_2), \sO_X(-h_1), \sO_X\}.
\end{equation}

From (\ref{col}) with a  left mutation of the pair  $\{\sO_X(-h_1), \sO_X\}$ we obtain:

 \begin{gather}\label{col01}
 \{ \sO_X(-h)[-4], \sO_X(-h_2-h_3)[-4], \sO_X(-h_1-h_3)[-3], \\
 \sO_X(-h_1-h_2)[-2], \sO_X(-h_3)[-2], \sO_X(-h_2)[-1], \sO_X(-2h_1), \sO_X(-h_1) \}. \notag
\end{gather}

From the above collection with a  left mutation of the pair  $\{\sO_X(-h_2)[-1], \sO_X(-2h_1) \}$ we obtain:

 \begin{gather}\label{col02}
 \{ \sO_X(-h)[-4], \sO_X(-h_2-h_3)[-4], \sO_X(-h_1-h_3)[-3], \\
 \sO_X(-h_1-h_2)[-2], \sO_X(-h_3)[-2],  A[-1], \sO_X(-h_2)[-1],  \sO_X(-h_1) \} \notag
\end{gather}

where $A$ is given by the extension

\begin{equation}\label{colA} 0\to\sO_X(-2h_1)\to A\to \sO_X(-h_2)^{\oplus 2}\to 0.
\end{equation}

From the above collection with a  left mutation of the pair  $\{\sO_X(-h_3), A \}$ we obtain:

 \begin{gather}\label{col2}
 \{ \sO_X(-h)[-4], \sO_X(-h_2-h_3)[-4], \sO_X(-h_1-h_3)[-3], \\
 \sO_X(-h_1-h_2)[-2], B[-2], \sO_X(-h_3)[-2], \sO_X(-h_2)[-1],  \sO_X(-h_1) \} \notag
\end{gather}

where $B$ is given by the extension

\begin{equation}\label{colB} 0 \to A\to B\to \sO_X(-h_3)^{\oplus 2}\to 0.
\end{equation}

Making the respective right mutation of \eqref{cold} we obtain the full exceptional collection $\langle F_7=\mathcal F_n, \ldots, F_0=\mathcal F_0\rangle$ of Theorem \ref{use}:

\begin{equation}\label{cold2}\begin{aligned}
\{ \sO_X(-h), \sO_X(-h_2-h_3), \sO_X(-h_1-h_3), \sO_X(-h_1-h_2), \sO_X, \sO_X(h_3), \sO_X(h_2), \sO_X(h_1) \}.
\end{aligned}\end{equation}

It is easy to check that the conditions \eqref{order} are satisfied. Observe that both collections \eqref{col} and \eqref{cold2} are strong.

\begin{theorem}\label{monadeteo}
Let $E$ be a charge $k$ instanton bundle on $X$ with $c_2(E)=k_1e_1+k_2e_2+k_3e_3$, then $E$ is the cohomology of a monad of the form
\begin{equation}\label{monade}
0\rightarrow
\begin{matrix}
\sO_X^{k_3}(-h_1-h_2) \\
\oplus \\
\sO_X^{k_2}(-h_1-h_3) \\
\oplus \\
\sO_X^{k_1}(-h_2-h_3)
\end{matrix}
\rightarrow
\begin{matrix}
\sO_X^{k_2+k_3}(-h_1) \\
\oplus \\
\sO_X^{k_1+k_3}(-h_2) \\
\oplus\\
\sO_X^{k_1+k_2}(-h_3)
\end{matrix}
\rightarrow
\sO_X^{k-2}
\rightarrow
0.
\end{equation}
Conversely any $\mu$-semistable bundle defined as the cohomology of such  a monad is a  charge $k$ instanton bundle.
\end{theorem}
\begin{proof}
We consider the Beilinson type spectral sequence associated to an instanton bundle $E$ and identify the members of the graded sheaf associated to the induced filtration as the sheaves mentioned in the statement of Theorem \ref{use}.
We consider the full exceptional collection $\langle E_7, \ldots, E_0\rangle$  given in \eqref{col} and the full exceptional collection $\langle F_7, \ldots, F_0\rangle$ given in \eqref{cold}.

First of all, let us observe that since $H^0(E)=0$ we have $H^0(E(-D))=0$ for every effective divisor $D$. Furthermore by Serre's duality we have also $H^2(E(K+D))=0$ for all effective divisors $D$.
Since $c_1(E)=0$ using Serre's duality and $H^1(E(-h))=0$ we obtain
\[
H^i(E(-h))=H^{3-i}(E(-h))=0 \ \text{for all $i$.}
\]
We want to show that for each twist in the table, there's only one non vanishing cohomology group, so that we can use the Riemann-Roch formula to compute the dimension of the remaining cohomology group.
Let us consider the pull-back of the Euler sequence from one of the $\p$ factors
\begin{equation}\label{eulerogenerale}
0\rightarrow \sO_{X}(-h_a) \rightarrow \sO_X^2 \rightarrow \sO_X(h_a) \rightarrow 0
\end{equation}
and tensor it by $E(-h)$. We have
\[
0\rightarrow E(-2h_a-h_b-h_c) \rightarrow E^2(-h) \rightarrow E(-h_b-h_c)\rightarrow 0
\]
with $a,b,c \in \{1,2,3\}$ and they are all different from each other. Since $H^i(E(-h))=0$ for all $i$ and $H^0(E(-2h_a-h_b-h_c))=H^3(E(-2h_a-h_b-h_c))=0$, considering the long exact sequence induced in cohomology we have $H^2(E(-h_b-h_c))=0$. Now we want to show that $H^2(E(-h_a))=0$ for all $a \in \{ 1,2,3\}$. Tensor \eqref{eulerogenerale} by $E(-h_b)$ with $b \neq a$ and we have:
\[
 0\rightarrow E(-2h_a-h_b) \rightarrow E^2(-h_a-h_b) \rightarrow E(-h_b)\rightarrow 0.
 \]
Considering the long exact sequence induced in cohomology we have that $H^2(E(-h_b))=0$ since $H^2(E(-h_a-h_b))=H^3(E(-2h_a-h_b))=0$. Finally if we tensor \eqref{eulerogenerale} by $E(-h_a)$ and we consider the long exact sequence in cohomology, we obtain $H^2(E)=0$.

Now let us compute the Euler characteristic of $E$ tensored by a line bundle $\sO_X(D)$ so that we are able to compute all the numbers in the Beilinson's table. Combining \eqref{chernppp} and \eqref{RR} we have
\begin{equation}\label{rrdivisor}
\chi (E(D))=\frac{1}{6}(2D^3-6c_2(E)D)+h(D^2-c_2(E)) +Dh^2+2.
\end{equation}
By \eqref{rrdivisor} we have
\begin{itemize}
\item
$h^1(E)=-\chi(E)=2-k_1-k_2-k_3=2-k$.
\item $h^1(E(-h_i))=-\chi(E(-h_i))=k_i-k$.
\item $h^1(E(-h_i-h_j))=-\chi(E(-h_i-h_j))=k_i+k_j-k$.
\end{itemize} So we get the following table:
{\tiny
\begin{center}
\renewcommand\arraystretch{2}
\begin{tabular}{ccccccccl}
\multicolumn{1}{c}{$\sO_X(-h)$}  & \multicolumn{1}{c}{$\sO_X(-h_2-h_3)$}  & \multicolumn{1}{c}{$\sO_X(-h_1-h_3)$} & \multicolumn{1}{c}{$\sO_X(-h_1-h_2)$} & \multicolumn{1}{c}{$\sO_X(-h_3)$} &  \multicolumn{1}{c}{$\sO_X(-h_2)$} & \multicolumn{1}{c}{$\sO_X(-h_1)$} & \multicolumn{1}{c}{$\sO_X$} &    \\ \cline{1-8}
\multicolumn{1}{|c|}{0} & \multicolumn{1}{c|}{0}        & \multicolumn{1}{c|}{0}        & \multicolumn{1}{c|}{0} & \multicolumn{1}{c|}{0}   & \multicolumn{1}{c|}{0}  & \multicolumn{1}{c|}{0}  & \multicolumn{1}{c|}{0}   & $h^7$ \\ \cline{1-8}
\multicolumn{1}{|c|}{0} & \multicolumn{1}{c|}{0}        & \multicolumn{1}{c|}{0}        & \multicolumn{1}{c|}{0} & \multicolumn{1}{c|}{0}   & \multicolumn{1}{c|}{0}  & \multicolumn{1}{c|}{0}  & \multicolumn{1}{c|}{0}   & $h^6$ \\ \cline{1-8}
\multicolumn{1}{|c|}{0} & \multicolumn{1}{c|}{$k_1$}        & \multicolumn{1}{c|}{0}        & \multicolumn{1}{c|}{0} & \multicolumn{1}{c|}{0}   & \multicolumn{1}{c|}{0}  & \multicolumn{1}{c|}{0}  & \multicolumn{1}{c|}{0}   & $h^5$ \\ \cline{1-8}
\multicolumn{1}{|c|}{0} & \multicolumn{1}{c|}{0}        & \multicolumn{1}{c|}{$k_2$}        & \multicolumn{1}{c|}{0} & \multicolumn{1}{c|}{0}   & \multicolumn{1}{c|}{0}  & \multicolumn{1}{c|}{0}  & \multicolumn{1}{c|}{0}   & $h^4$ \\ \cline{1-8}
\multicolumn{1}{|c|}{0} & \multicolumn{1}{c|}{0}        & \multicolumn{1}{c|}{0}        & \multicolumn{1}{c|}{$k_3$} & \multicolumn{1}{c|}{$k_1+k_2$}   & \multicolumn{1}{c|}{0}  & \multicolumn{1}{c|}{0}  & \multicolumn{1}{c|}{0}   & $h^3$ \\ \cline{1-8}
\multicolumn{1}{|c|}{0} & \multicolumn{1}{c|}{0}        & \multicolumn{1}{c|}{0}        & \multicolumn{1}{c|}{0} & \multicolumn{1}{c|}{0}   & \multicolumn{1}{c|}{$k_1+k_3$}  & \multicolumn{1}{c|}{0}  & \multicolumn{1}{c|}{0}   & $h^2$ \\ \cline{1-8}
\multicolumn{1}{|c|}{0} & \multicolumn{1}{c|}{0}        & \multicolumn{1}{c|}{0}        & \multicolumn{1}{c|}{0} & \multicolumn{1}{c|}{0}   & \multicolumn{1}{c|}{0}  & \multicolumn{1}{c|}{$k_2+k_3$}  & \multicolumn{1}{c|}{$k-2$}   & $h^1$ \\ \cline{1-8}
\multicolumn{1}{|c|}{0} & \multicolumn{1}{c|}{0}        & \multicolumn{1}{c|}{0}        & \multicolumn{1}{c|}{0} & \multicolumn{1}{c|}{0}   & \multicolumn{1}{c|}{0}  & \multicolumn{1}{c|}{0}  & \multicolumn{1}{c|}{0}   & $h^0$ \\ \cline{1-8}
\multicolumn{1}{c}{ $E(-h)[-4]$}  & \multicolumn{1}{c}{$E(-h_2-h_3)[-4]$}  & \multicolumn{1}{c}{$E(-h_1-h_3)[-3]$}  & \multicolumn{1}{c}{$E(-h_1-h_2)[-2]$}  & \multicolumn{1}{c}{ $E(-h_3)[-2]$} & \multicolumn{1}{c}{ $E(-h_2)[-1]$} & \multicolumn{1}{c}{ $E(-h_1)$} & \multicolumn{1}{c}{ $E$} &
\end{tabular}
\label{tabella2}
\end{center}}

Using Beilinson's theorem in the strong form (as in Remark \ref{rembeil}) we retrieve the monad \eqref{monade}.

Conversely let $E$ be a $\mu$-semistable bundle defined as the cohomology of a monad (\ref{monade}). We may consider the two short exact sequences:
\begin{equation}\label{m1}
0\rightarrow G \rightarrow
\begin{matrix}
\sO_X^{k_2+k_3}(-h_1) \\
\oplus \\
\sO_X^{k_1+k_3}(-h_2) \\
\oplus\\
\sO_X^{k_1+k_2}(-h_3)
\end{matrix}
\rightarrow \sO_X^{k-2} \rightarrow 0
\end{equation}
and
\begin{equation}\label{m2}
0 \rightarrow
\begin{matrix}
\sO_X^{k_3}(-h_1-h_2) \\
\oplus \\
\sO_X^{k_2}(-h_1-h_3) \\
\oplus \\
\sO_X^{k_1}(-h_2-h_3)
\end{matrix}
\rightarrow
G
\rightarrow
E
\rightarrow 0.
\end{equation}

We deduce that $H^0(G)=H^0(E)=0$.
By (\ref{m1}) and (\ref{m2}) tensored by $\sO_X(-h)$ we obtain $H^1(G(-h))=H^1(E(-h))=0$ so $E$ is an instanton.
\end{proof}

\begin{prop}
Let $E$ be an instanton bundle on $X$, then $h^1(E(-h-D))=0$ for every effective divisor $D$.
\end{prop}
\begin{proof}
Let us consider the two short exact sequences  \eqref{m1} and \eqref{m2} tensored by $\sO_X(-h+D)$.
By Kunneth formula we have that $h^i(\sO_X(-h-D))=0$ for all $i$, and thus taking the cohomology of \eqref{m1} we get $h^i(G(-h-D))=0$ for $i\neq 3$. Combining this with the induced sequence in cohomology of \eqref{m2} we obtain $h^0(E(-h-D))=h^1(E(-h-D))=0$.
\end{proof}

In the next theorem we obtain a description of instanton bundles as the cohomology of a different monad.

\begin{theorem}\label{monadeteo2}
Let $E$ be a charge $k$ instanton bundle on $X$ with $c_2(E)=k_1e_1+k_2e_2+k_3e_3$, then $E$ is the cohomology of a monad of the form
\begin{equation}\label{monade2}
0\rightarrow
\begin{matrix}
\sO_X^{k_3}(-h_1-h_2) \\
\oplus \\
\sO_X^{k_2}(-h_1-h_3) \\
\oplus \\
\sO_X^{k_1}(-h_2-h_3)
\end{matrix}
\rightarrow
\begin{matrix}
\sO_X^{3k+2}
\rightarrow
\end{matrix}
\begin{matrix}
\sO_X^{k_2+k_3}(h_1) \\
\oplus \\
\sO_X^{k_1+k_3}(h_2) \\
\oplus\\
\sO_X^{k_1+k_2}(h_3)
\end{matrix}
\rightarrow
0.
\end{equation}
Conversely any $\mu$-semistable bundle with $H^0(E)=0$ defined as the cohomology of such  a monad is a  charge $k$ instanton bundle.
\end{theorem}

\begin{proof}
We consider the Beilinson type spectral sequence associated to an instanton bundle $E$ and identify the members of the graded sheaf associated to the induced filtration as the sheaves mentioned in the statement of Theorem \ref{use}.
We consider the full exceptional collection $\langle E_7, \ldots, E_0\rangle$  given in \eqref{col2} and the full exceptional collection $\langle F_7, \ldots, F_0\rangle$ given in \eqref{cold2}.

First of all, let us observe that since since $E$ is $\mu$-semistable, by Hoppe's criterion we have $H^0(E(-D))=0$ for every effective divisor $D$. Furthermore  we have all the vanishing computed in Theorem \ref{monadeteo}
Moreover by \eqref{colA} and \eqref{colB} tensored by $E$ we get
$\chi(E\otimes B)=\chi(E\otimes A)+2\chi(E(-h_3))=\chi(E(-2h_1))+2\chi(E(-h_3))+ 2\chi(E(-h_2))=-2+k_1-k_2-k_3-2(k_1+k_2)-2(k_1+k_3)=-2-3k$.
 So we get the following table:
{\tiny
\begin{center}
\renewcommand\arraystretch{2}
\begin{tabular}{ccccccccl}
\multicolumn{1}{c}{$\sO_X(-h)$}  & \multicolumn{1}{c}{$\sO_X(-h_2-h_3)$}  & \multicolumn{1}{c}{$\sO_X(-h_1-h_3)$} & \multicolumn{1}{c}{$\sO_X(-h_1-h_2)$} & \multicolumn{1}{c}{$\sO_X(-h_3)$} &  \multicolumn{1}{c}{$\sO_X(-h_2)$} & \multicolumn{1}{c}{$\sO_X(-h_1)$} & \multicolumn{1}{c}{$\sO_X$} &    \\ \cline{1-8}
\multicolumn{1}{|c|}{0} & \multicolumn{1}{c|}{0}        & \multicolumn{1}{c|}{0}        & \multicolumn{1}{c|}{0} & \multicolumn{1}{c|}{0}   & \multicolumn{1}{c|}{0}  & \multicolumn{1}{c|}{0}  & \multicolumn{1}{c|}{0}   & $h^7$ \\ \cline{1-8}
\multicolumn{1}{|c|}{0} & \multicolumn{1}{c|}{0}        & \multicolumn{1}{c|}{0}        & \multicolumn{1}{c|}{0} & \multicolumn{1}{c|}{0}   & \multicolumn{1}{c|}{0}  & \multicolumn{1}{c|}{0}  & \multicolumn{1}{c|}{0}   & $h^6$ \\ \cline{1-8}
\multicolumn{1}{|c|}{0} & \multicolumn{1}{c|}{$k_1$}        & \multicolumn{1}{c|}{0}        & \multicolumn{1}{c|}{0} & \multicolumn{1}{c|}{0}   & \multicolumn{1}{c|}{0}  & \multicolumn{1}{c|}{0}  & \multicolumn{1}{c|}{0}   & $h^5$ \\ \cline{1-8}
\multicolumn{1}{|c|}{0} & \multicolumn{1}{c|}{0}        & \multicolumn{1}{c|}{$k_2$}        & \multicolumn{1}{c|}{0} & \multicolumn{1}{c|}{a}   & \multicolumn{1}{c|}{0}  & \multicolumn{1}{c|}{0}  & \multicolumn{1}{c|}{0}   & $h^4$ \\ \cline{1-8}
\multicolumn{1}{|c|}{0} & \multicolumn{1}{c|}{0}        & \multicolumn{1}{c|}{0}        & \multicolumn{1}{c|}{$k_3$} & \multicolumn{1}{c|}{b}   & \multicolumn{1}{c|}{$k_1+k_2$}  & \multicolumn{1}{c|}{0}  & \multicolumn{1}{c|}{0}   & $h^3$ \\ \cline{1-8}
\multicolumn{1}{|c|}{0} & \multicolumn{1}{c|}{0}        & \multicolumn{1}{c|}{0}        & \multicolumn{1}{c|}{0} & \multicolumn{1}{c|}{0}   & \multicolumn{1}{c|}{0}  & \multicolumn{1}{c|}{$k_1+k_3$}  & \multicolumn{1}{c|}{0}   & $h^2$ \\ \cline{1-8}
\multicolumn{1}{|c|}{0} & \multicolumn{1}{c|}{0}        & \multicolumn{1}{c|}{0}        & \multicolumn{1}{c|}{0} & \multicolumn{1}{c|}{0}   & \multicolumn{1}{c|}{0}  & \multicolumn{1}{c|}{0}  & \multicolumn{1}{c|}{$k_2+k_3$}   & $h^1$ \\ \cline{1-8}
\multicolumn{1}{|c|}{0} & \multicolumn{1}{c|}{0}        & \multicolumn{1}{c|}{0}        & \multicolumn{1}{c|}{0} & \multicolumn{1}{c|}{0}   & \multicolumn{1}{c|}{0}  & \multicolumn{1}{c|}{0}  & \multicolumn{1}{c|}{0}   & $h^0$ \\ \cline{1-8}
\multicolumn{1}{c}{ $E(-h)[-4]$}  & \multicolumn{1}{c}{$E(-h_2-h_3)[-4]$}  & \multicolumn{1}{c}{$E(-h_1-h_3)[-3]$}  & \multicolumn{1}{c}{$E(-h_1-h_2)[-2]$} & \multicolumn{1}{c}{ $E\otimes B[-2]$} & \multicolumn{1}{c}{ $E(-h_3)[-2]$} & \multicolumn{1}{c}{ $E(-h_2)[-1]$} & \multicolumn{1}{c}{ $E(-h_1)$}  &
\end{tabular}
\label{tabella3}
\end{center}}
where $a-b=-2-3k$. Since the spectral sequence converges to an object in degree $0$ and there no maps involving $a$ we deduce that $a=0$ and $b=3k+2$.
So we get the following table:
{\tiny
\begin{center}
\renewcommand\arraystretch{2}
\begin{tabular}{ccccccccl}
\multicolumn{1}{c}{$\sO_X(-h)$}  & \multicolumn{1}{c}{$\sO_X(-h_2-h_3)$}  & \multicolumn{1}{c}{$\sO_X(-h_1-h_3)$} & \multicolumn{1}{c}{$\sO_X(-h_1-h_2)$} & \multicolumn{1}{c}{$\sO_X(-h_3)$} &  \multicolumn{1}{c}{$\sO_X(-h_2)$} & \multicolumn{1}{c}{$\sO_X(-h_1)$} & \multicolumn{1}{c}{$\sO_X$} &    \\ \cline{1-8}
\multicolumn{1}{|c|}{0} & \multicolumn{1}{c|}{0}        & \multicolumn{1}{c|}{0}        & \multicolumn{1}{c|}{0} & \multicolumn{1}{c|}{0}   & \multicolumn{1}{c|}{0}  & \multicolumn{1}{c|}{0}  & \multicolumn{1}{c|}{0}   & $h^7$ \\ \cline{1-8}
\multicolumn{1}{|c|}{0} & \multicolumn{1}{c|}{0}        & \multicolumn{1}{c|}{0}        & \multicolumn{1}{c|}{0} & \multicolumn{1}{c|}{0}   & \multicolumn{1}{c|}{0}  & \multicolumn{1}{c|}{0}  & \multicolumn{1}{c|}{0}   & $h^6$ \\ \cline{1-8}
\multicolumn{1}{|c|}{0} & \multicolumn{1}{c|}{$k_1$}        & \multicolumn{1}{c|}{0}        & \multicolumn{1}{c|}{0} & \multicolumn{1}{c|}{0}   & \multicolumn{1}{c|}{0}  & \multicolumn{1}{c|}{0}  & \multicolumn{1}{c|}{0}   & $h^5$ \\ \cline{1-8}
\multicolumn{1}{|c|}{0} & \multicolumn{1}{c|}{0}        & \multicolumn{1}{c|}{$k_2$}        & \multicolumn{1}{c|}{0} & \multicolumn{1}{c|}{0}   & \multicolumn{1}{c|}{0}  & \multicolumn{1}{c|}{0}  & \multicolumn{1}{c|}{0}   & $h^4$ \\ \cline{1-8}
\multicolumn{1}{|c|}{0} & \multicolumn{1}{c|}{0}        & \multicolumn{1}{c|}{0}        & \multicolumn{1}{c|}{$k_3$} & \multicolumn{1}{c|}{3k+2}   & \multicolumn{1}{c|}{$k_1+k_2$}  & \multicolumn{1}{c|}{0}  & \multicolumn{1}{c|}{0}   & $h^3$ \\ \cline{1-8}
\multicolumn{1}{|c|}{0} & \multicolumn{1}{c|}{0}        & \multicolumn{1}{c|}{0}        & \multicolumn{1}{c|}{0} & \multicolumn{1}{c|}{0}   & \multicolumn{1}{c|}{0}  & \multicolumn{1}{c|}{$k_1+k_3$}  & \multicolumn{1}{c|}{0}   & $h^2$ \\ \cline{1-8}
\multicolumn{1}{|c|}{0} & \multicolumn{1}{c|}{0}        & \multicolumn{1}{c|}{0}        & \multicolumn{1}{c|}{0} & \multicolumn{1}{c|}{0}   & \multicolumn{1}{c|}{0}  & \multicolumn{1}{c|}{0}  & \multicolumn{1}{c|}{$k_2+k_3$}   & $h^1$ \\ \cline{1-8}
\multicolumn{1}{|c|}{0} & \multicolumn{1}{c|}{0}        & \multicolumn{1}{c|}{0}        & \multicolumn{1}{c|}{0} & \multicolumn{1}{c|}{0}   & \multicolumn{1}{c|}{0}  & \multicolumn{1}{c|}{0}  & \multicolumn{1}{c|}{0}   & $h^0$ \\ \cline{1-8}
\multicolumn{1}{c}{ $E(-h)[-4]$}  & \multicolumn{1}{c}{$E(-h_2-h_3)[-4]$}  & \multicolumn{1}{c}{$E(-h_1-h_3)[-3]$}  & \multicolumn{1}{c}{$E(-h_1-h_2)[-2]$} & \multicolumn{1}{c}{ $E\otimes B[-2]$} & \multicolumn{1}{c}{ $E(-h_3)[-2]$} & \multicolumn{1}{c}{ $E(-h_2)[-1]$} & \multicolumn{1}{c}{ $E(-h_1)$}  &
\end{tabular}
\label{tabella1}
\end{center}}

Using Beilinson's theorem as in Remark \ref{rembeil} we retrieve the monad \eqref{monade2}.

Conversely let $E$ be a $\mu$-semistable bundle with no global sections defined as the cohomology of a monad \eqref{monade}. We may consider the two short exact sequences:
\begin{equation}\label{m11}
0\rightarrow G \rightarrow
\sO_X^{3k+2}
\rightarrow
\begin{matrix}
\sO_X^{k_2+k_3}(h_1) \\
\oplus \\
\sO_X^{k_1+k_3}(h_2) \\
\oplus\\
\sO_X^{k_1+k_2}(h_3)
\end{matrix}
\rightarrow 0
\end{equation}
and
\begin{equation}\label{m22}
0 \rightarrow
\begin{matrix}
\sO_X^{k_3}(-h_1-h_2) \\
\oplus \\
\sO_X^{k_2}(-h_1-h_3) \\
\oplus \\
\sO_X^{k_1}(-h_2-h_3)
\end{matrix}
\rightarrow
G
\rightarrow
E
\rightarrow 0.
\end{equation}

By \eqref{m11} and \eqref{m22} tensored by $\sO_X(-h)$ we obtain $H^1(G(-h))=H^1(E(-h))=0$ so $E$ is an instanton.
\end{proof}

\begin{remark}
It is possible to construct vector bundles which are realized as the cohomology of a monad as in Theorem \ref{monadeteo} and \ref{monadeteo2} but that are not $\mu$-semistable. Let us consider a generic line $l$ in the ruling $e_1$. It has the following resolution on $X$
\begin{equation}\label{resolutionline}
0 \to \sO_X(-h_2-h_3) \to \sO_X(-h_2) \oplus \sO_X(-h_3) \to \sO_X \to \sO_{l} \to 0.
\end{equation}
By adjunction formula we have $\sN_{l/X}^\vee \cong \sI_{l|X}\otimes \sO_l$, and using \eqref{resolutionline} we obtain $\sN_{l/X}^\vee \cong \sO_l^2$ and in particular $\det \sN_{l/X} \otimes \sO_l \cong \sO_X(D)\otimes \sO_l$ where $D$ is a divisor of the form $D=ah_2+bh_3$. Choosing $D=2h_2-4h_3$, since $h^2(\sO_X(-D))=0$, it is possible to construct a vector bundle $E$ with $c_1(E)=0$ and $c_2(E)=5e_1$ through the Hartshorne-Serre correspondence (for details see \cite{arrondo}) which fits into
\begin{equation}\label{serrecontroesempio}
0\to \sO_X(-h_2+2h_3)\to E \to \sI_{l|X}(h_2-2h_3)\to 0.
\end{equation}
The vector bundle constructed in this way has no sections, i.e. $H^0(E)=0$ and if we tensor \eqref{serrecontroesempio} by $\sO_X(-h)$ and we take the cohomology, we obtain
\[
H^i(E(-h))\cong H^i(\sI_{l|X}(-h_1-3h_3)).
\]
Now consider the sequence
\begin{equation}\label{idealseq}
0\to \sI_{l|X}\to \sO_X\to \sO_{l}\to 0.
\end{equation}
Tensoring \eqref{idealseq} by $\sO_X(-h_1-3h_3)$, we get $h^1(\sI_{l|X}(-h_1-3h_3))= h^0(\sO_l(-h_1-3h_3))= h^0(\sO_\p(-1))=0$. Thus we obtain $H^1(E(-h))=0$. In this way we constructed a vector bundle $E$ with $c_1(E)=0$ satisfying all the instantonic conditions but the $\mu$-semistability. In fact by Proposition \ref{hoppe} $E$ is not $\mu$-semistable since $H^0(E(h_2-2h_3))\neq 0$. Furthermore $E$ has the same cohomology table of an instanton bundle, thus it is realized as the cohomology of the monads
\[
0 \to \sO_X^5(-h_2-h_3) \to \sO_X^5(-h_2) \oplus \sO_X^5(-h_3) \to \sO_X^3 \to 0
\] 
and 
\[
0 \to \sO_X^5(-h_2-h_3) \to \sO_X^{17}\to \sO_X^5(h_2)\oplus \sO_X^5(h_3) \to 0.
\]
\end{remark}
\begin{remark}
Let us remark that the monad \eqref{monade2} is the analog of the monad for instanton bundles on $\mathbb P^3$
$$
0 \to \sO_{\mathbb P^3}(-1)^{\oplus k} \xrightarrow{\alpha}  \sO_{\mathbb P^3}^{\oplus 2k+2}
\xrightarrow{\beta} \sO_{\mathbb P^3}(1)^{\oplus k} \to 0,$$
and the monad (\ref{monade}) is the analog of the second monad for instanton bundles on $\mathbb P^3$
(see for instance \cite{AO} display $(1.1)$)
$$
0 \to \sO_{\mathbb P^3}(-1)^{\oplus k} \xrightarrow{\alpha}  \Omega_{\mathbb P^3}(1)^{\oplus k}
\xrightarrow{\beta} \sO_{\mathbb P^3}^{\oplus 2k-2} \to 0.$$
A very similar behavior was shown for the two monads for instanton bundles on the flag threefold in \cite{MMP}.

As in the case of instanton bundles on the projective space and flag varieties, the two monads \eqref{monade2} and \eqref{monade} are closely related. Indeed, sequence \eqref{m11} fits in the following commutative diagram
$$
\xymatrix{
& 0 \ar[d] & 0  \ar[d] & 0 \ar[d] \\
0 \ar[r] & G \ar[d] \ar[r] & \sO^{\oplus 3k+2} \ar[d] \ar[r]^(.3){\beta} &  \oplus_{i\in\mathbb Z_3}\sO_X^{k_{i+1}+k_{i+2}}(h_i) \ar[r] \ar[d] & 0\\
0 \ar[r] & \oplus_{i\in\mathbb Z_3}\sO_X^{k_{i+1}+k_{i+2}}(-h_i) \ar[d] \ar[r] & \sO^{\oplus 4k} \ar[d] \ar[r]^(.3){\beta'} &  \oplus_{i\in\mathbb Z_3}\sO_X^{k_{i+1}+k_{i+2}}(h_i) \ar[r] \ar[d] & 0\\
0 \ar[r] & \sO^{\oplus k-2} \ar[d] \ar[r] & \sO^{\oplus k-2} \ar[d] \ar[r] & 0 \\
& 0 & 0
}
$$
 So we get sequence \eqref{m1} as the first column. Moreover sequence \eqref{m22} fits in the following commutative diagram
$$
\xymatrix{
& 0 \ar[d] & 0 \ar[d] & 0 \ar[d] \\
0 \ar[r] & \oplus_{i\in\mathbb Z_3}\sO_X^{k_{i}}(-h_{i+1}-h_{i+2}) \ar[r] \ar[d] & G \ar[r] \ar[d] & E \ar[r] \ar[d] & 0 \\
0 \ar[r] & \oplus_{i\in\mathbb Z_3}\sO_X^{k_{i}}(-h_{i+1}-h_{i+2}) \ar[r] \ar[d] & \oplus_{i\in\mathbb Z_3}\sO_X^{k_{i+1}+k_{i+2}}(-h_i) \ar[r] \ar[d] & C \ar[r] \ar[d] & 0 \\
 & 0\ar[r] & \sO^{\oplus k-2}\ar[r]\ar[d] & \sO^{\oplus k-2} \ar[r] \ar[d]& 0\\
 & & 0 & 0
}
$$
which is the display of monad \eqref{monade}.

Finally, for the monad \eqref{monade2} is not necessary the assumption $H^0(E)=0$. Exactly the same behavior was shown for the analog monad on $F(0,1,2)$ (see \cite{MMP} Theorem 4.2). 
\end{remark}

We end this section by characterizing the strictly Gieseker semistable instanton bundles on $X$
\begin{prop}\label{semistable}
Let $E$ be an instanton bundle of charge $k$. If $E$ is not $\mu$-stable then $k=2l^2$ for some $l\in \mathbb{Z}$, $l \neq 0$. Moreover $c_2(E)=2l^2e_i$, $i=1,2,3$ and $E$ can be constructed as an extension
\begin{equation}\label{ext1}
0\rightarrow \sO_X(-lh_i+lh_j) \rightarrow E \rightarrow \sO_X(lh_i-lh_j)\rightarrow 0
\end{equation}
with $i \neq j$.
\end{prop}
\begin{proof}
Suppose $H^0(X,E(ah_1+bh_2-(a+b)h_3))\neq 0$ for some $a,b\in \mathbb{Z}$. So $E$ fits into an exact sequence
\[
0 \rightarrow \sO_X \rightarrow  E(ah_1+bh_2-(a+b)h_3) \rightarrow \mathcal{I}_Z(2ah_1+2bh_2-2(a+b)h_3)\rightarrow 0
\]
where $Z\subset X$ is a subscheme of $X$. Since $H^0(E(ah_1+bh_2-(a+b)h_3)\otimes \sO_X(-h_j))=0$ for all $j=1,2,3$ by Proposition \ref{hoppe}, we have that $Z\subset X$ is either empty or purely 2-codimensional. Suppose we are dealing with the latter case, since $E$ in Gieseker semistable we have that
\[
P_{\sO_X}(t)\leq P_{E(ah_1+bh_2-(a+b)h_3)}(t) \leq P_{\mathcal{I}_Z(2ah_1+2bh_2-2(a+b)h_3)}(t)
\]
and
\[
P_{\mathcal{I}_Z(2ah_1+2bh_2-2(a+b)h_3)}(t)=P_{\sO_X(2ah_1+2bh_2-2(a+b)h_3)}(t)-P_{\sO_Z(2ah_1+2bh_2-2(a+b)h_3)}(t)
\]
where $P(t)$ is the Hilbert polynomial. So we have
\begin{align*}
P_{\sO_Z(2ah_1+2bh_2-2(a+b)h_3)}(t) & \leq P_{\sO_X(2ah_1+2bh_2-2(a+b)h_3)}(t)-P_{\sO_X}(t) \\
&= (2a+t+1)(2b+t+1)(t+1-2a-2b)-(t+1)^3 \\
&= -4(t+1)(a^2+b^2+ab) <0 \ \text{for $t>>0$}.
\end{align*}
contradicting Serre's vanishing theorem. Se we can conclude that $Y$ is empty and $E$ fits into
\[
0 \rightarrow \sO_X(-ah_1-bh_2+(a+b)h_3) \rightarrow E \rightarrow \sO_X(ah_1+bh_2-(a+b)h_3)\rightarrow 0.
\]
Now computing $c_2(E)$ we obtain
\begin{align*}
c_2(E)&=(-ah_1-bh_2+(a+b)h_3)\cdot (ah_1+bh_2-(a+b)h_3) \\
&=2b(a+b)e_1+2a(a+b)e_2-2abe_3.
\end{align*}
Since $E$ is an instanton bundle on $X$, all the summands of $c_2(E)$ must be nonnegative. In fact as we will see in Proposition \ref{monadeteo} they represent the dimension of a cohomology group. So either $a$ or $b$ is 0 (but not both since the charge $k$ must be greater than two) or $a=-b$. In all three cases we obtain the desired result.
\end{proof}

\section{Splitting behaviour of Ulrich bundles}
In the next sections we will construct, through an induction process, stable k-instanton bundles on $X$ for each charge k and all second Chern classes. Let us begin with the following remark which explain the relation between Ulrich and instanton bundles.

\begin{remark}
Notice that when the charge is minimal, namely $k=2$, any instanton bundle is, up to twist, an Ulrich bundle. From \cite{pppacm} we get that any rank two Ulrich bundle with $c_1=2h$ is an instanton bundle twisted by $h$. However the rank two Ulrich bundle with $c_1=h_1+2h_2+3h_3$ cannot be the twist of an instanton bundle. On the two Fano threefold with index two and Picard number two all the instanton bundle with minimal degree of $c_2(E)$ is a twist of an Ulrich bundle (see \cite{MMP} and \cite{CCGM}). If $E$ is an instanton bundle on a Fano threefold $X$ with index two and Picard number one we have $h^1(E)=c_2(E)-2$ and $h^0(E)=h^2(E)=h^3(E)=0$ (see section 3 of \cite{Fa2}). So when the charge is minimal, namely $c_2(E)=2$, by Serre duality we have $h^i(E)=h^i(E(-1))=h^i(E(-2))=0$ for any $i$ so $E$ is Ulrich up to a twist. On the three dimensional quadric (the only case of index three) the instanton bundle of minimal charge is the spinor bundle which is Ulrich.   The  monad for instanton bundles on $\mathbb P^3$ is
$$
0 \to \sO_{\mathbb P^3}(-1)^{\oplus k} \xrightarrow{\alpha}  \sO_{\mathbb P^3}^{\oplus 2k+2}
\xrightarrow{\beta} \sO_{\mathbb P^3}(1)^{\oplus k} \to 0,$$
so when $k=0$ we get $E=\sO_{\mathbb P^3}^{\oplus 2}$. This is not an instanton bundle because is not simple. However, from a monadic point of view, we may say that the trivial bundle $E=\sO_{\mathbb P^3}^{\oplus 2}$, which is Ulrich, is a limit case.
The case of index one is much more complicated (see section 4 of \cite{Fa2}).
\end{remark}

Let us consider the base case of induction, which consists of charge 2 instantons on $X$, i.e. rank two Ulrich bundles (up to twisting by $\sO_X(-h)$). For further details about Ulrich bundles on $\ppp$ see \cite{p1p1p1}. We have two possible alternatives for the second Chern class of an Ulrich bundle:
\begin{itemize}
\item [ (a)] $c_2(E)= 2e_i$ for some $i \in \{1,2,3\}$.
\item [ (b)] $c_2(E)=e_i+e_j$ with $i\neq j$.
\end{itemize}
We show that in both in cases the generic Ulrich bundle has trivial restriction with respect to a generic line of each family. In both cases we have $\Ext^2(E,E)=\Ext^3(E,E)=0$ by \cite[Lemma 2.3]{p1p1p1}.
\begin{prop}
The generic instanton bundle of minimal charge $k=2$ has trivial restriction with respect to the generic line of each family $|e_1|$, $|e_2|$ and $|e_3|$.
\end{prop}
\begin{proof}
We will separate the proof analyzing both cases $(a)$ and $(b)$. 

{\bf{Case (a)}}

Let us begin with the first case. By Theorem \ref{monadeteo} we see that every rank two Ulrich bundle with this second Chern class is the pullback from a quadric $Q=\p \times \p$. In this case, by Proposition \ref{semistable}, there exist strictly semistable Ulrich bundle realized as extensions
\begin{equation}\label{estensioneUlrich}
0 \rightarrow \sO_X(h_j-h_k) \rightarrow E \rightarrow \sO_X(h_k-h_j)\rightarrow 0
\end{equation}
with $j\neq k \neq i \neq j$. For these vector bundles, by restricting \eqref{estensioneUlrich} to a line in each family, we observe that in the family $h_jh_k$ there are not jumping lines, i.e. $E_l=\sO_l^2$ for each $l\in \left| h_jh_k\right|$. On the other hand, $E_l$ is never trivial when $l \in \left| h_ih_k\right|$ or $l\in \left| h_jh_i\right|$. However the generic bundle will be stable, so let us focus on stable Ulrich bundles. They are pull back via the projection on the quadric, of stable bundles on $Q$. By \cite[Lemma 2.5]{chavez} every such bundle can be deformed to a stable bundle which is trivial when restricted to the generic line of each family.

{\bf{Case (b)}}

Now let us consider the second case. The details of what follows can be found in \cite{pppacm}. Up to a permutation of the indices we can assume $c_2(E)=e_2+e_3$. Let us denote by $H$ a general hyperplane section in $\mathbb{P}^7$ and let $S$ be $S=X\cap H$. $S$ is a del Pezzo surface of degree 6, given as the blow up of $\mathbb{P}^2$ in 3 points. Let us denote by $F$ the pullback to $S$ of the class of a line in $\mathbb{P}^2$ and by $E_i$ the exceptional divisors. Take a general curve $C$ of class $3F-E_1$, so that $C$ is a smooth, irreducible, elliptic curve of degree $8$. Moreover we have $h^0(C,\mathcal{N}_{C|X})=16$ and $h^1(C,\mathcal{N}_{C|X})=0$, so the Hilbert scheme $\mathscr{H}=\mathscr{H}^{8t}$ of degree 8 elliptic curves is smooth of dimension 16 \cite[Proposition 6.3]{pppacm} and the general deformation of $C$ in $\mathscr{H}$ is non-degenerate \cite[Proposition 6.6]{pppacm}. Let $\mathcal{C} \subset X \times B \rightarrow B$ a flat family of curves in $\mathscr{H}$ with special fibre $\mathcal{C}_{b_0}\cong C$ over $b_0$. To each curve in the family $\mathcal{C}$ we can associate a rank two vector bundle via the Serre's correspondence:
\begin{equation}\label{serreulrich}
0 \rightarrow \sO_X(-h) \rightarrow E_b \rightarrow I_{{C_b}|X}(h) \rightarrow 0
\end{equation}
where $C_b$ is the curve in $\mathcal{C}$ over $b\in B$. The general fiber $C_b$ correspond via \eqref{serreulrich} to rank two Ulrich bundle of the desired $c_2$.

Now choose a line $L$ in $S$, such that $L\cap C$ is a single point $x$. In order to do so, we deal with the classes of $F$ and $E_i$ in $A^2(X)$. One obtain that the classes of $F$, $E_1$, $E_2$ and $E_3$ are $e_1+e_2+e_3$, $e_1$, $e_2$ and $e_3$ respectively. In particular, there exists a line $L$ in the system $|E_1|$ (corresponding to $|e_1|$ in $A^2(X)$) which it intersect the curve $C$ in the class $3F-E_1$ in one point. It follows that $I_{C|X}(1)\otimes \mathcal{O}_L \cong \sO_x \oplus \sO_L$. Tensoring \eqref{serreulrich} by $\sO_L$ we obtain a surjection
\[
E_{{b_0}_{|L}} \rightarrow  \sO_x \oplus \sO_L\rightarrow 0.
\]
In particular $E_{{b_0}_{|L}}$ cannot be $\sO_L(-t) \oplus \sO_L(t)$ for any $t>0$, thus $E_{{b_0}_{|L}}$ is trivial, which is equivalent to $h^0(L,E_{{b_0}_{|L}}(-1))=0$. By semicontinuity we have that $h^0(L,E_{{b}_{|L}}(-1))=0$ for all $b$ in an open neighborhood of $b_0 \in B$, thus the vector bundle corresponding to the general fiber $C_b$ is trivial over the line $L$. Since this is an open condition on the variety of lines contained in $X$, it takes place for the general line in $|e_1|$.

To deal with the other families of lines let us consider a general quadric $Q$ in $|h_1|$. Let $C$ be a smooth, irreducible, non-degenerate elliptic curve in the class $2e_1+3e_2+3e_3$. $\Pic(Q)\cong\mathbb{Z}^2$ generated by two lines $<l,m>$ which correspond respectively to $e_3$ and $e_2$. Since $Q$ is general then $Z=C\cap Q$ consist of two points. Following the previous strategy, we say that $E$ restricted to a generic line of the family $e_2$ (resp. $e_3$) is trivial if $Z$ is not contained in a line of the ruling $m$ (resp. $l$). As in the previous case, let us consider the del Pezzo surface $S=X\cap H$ with $H$ a general hyperplane section. The intersection between $Q$ and $S$ is a curve in the class $e_2+e_3\in A^2(X)$. Let us denote by $C$ the curve $Y=Q\cap S$. We compute the class of $Y$ in $S$. We have the following short exact sequences
\begin{gather*}
0 \to \sO_X(-2h_1-h_2-h_3) \to \sO_X(-h_1)\oplus \sO_X(-h) \to \sO_X \to I_Y \to 0, \\
0\to I_C \to \sO_X \to \sO_Y \to 0
\end{gather*}
and computing the cohomology we find that $h^1(Y,\sO_Y)=g=0$. In particular we have that $Y$ is a degree two curve of genus $0$ on $S$, thus it must be in the class of $F-E_1$. Furthermore, observe that every line of each ruling of $Q$ intersect $S$ in only one point. Now let us take a general curve $C$ in the class $3F-E_1$, so that $C$ is a smooth, irreducible, elliptic curve of degree $8$. Computing the intersection product between $C$ and $Y$, we see that $C\cap Y$ consists of two points. Those two points cannot lie on a line in $Q$, because each line in $Q$ intersects $S$ only in one point. As before, let $\mathcal{C} \subset X \times B \rightarrow B$ a flat family of curves in $\mathcal{H}$ with special fibre $\mathcal{C}_{b_0}\cong C$ over $b_0$. To each curve in the family $\mathcal{C}$ we can associate a rank two vector bundle via the sequence \eqref{serreulrich}. Let $Z_b=C_b \cap Q$ and denote by $l$ and $m$ the two rulings of $Q$. We observed that $Z_{b_0}$ is not contained in a line either of $l$ or $m$, i.e $C_{b_0}$ intersects the generic line of both $|l|$ and $|m|$ in one point. But the rulings of $Q$ correspond to the rulings $e_2$ and $e_3$ of $X$, thus we can repeat the same argument used for the generic line in $e_1$. In this way we conclude that the vector bundle corresponding to the general fiber $C_b$ is trivial over the generic line of each of the families $|e_1|$, $|e_2|$ and $|e_3|$.
\end{proof}

\section{Construction of instanton bundles of higher charge}
In this section we will construct instanton bundles of every charge generically trivial on lines, through an induction process starting from Ulrich bundles. By doing so, we will also construct a nice component of the moduli space $MI(c_2)$ of instanton bundles on $X$ with fixed $c_2$. 

\begin{theorem}
For each non-negative $k_1,k_2,k_3 \in \mathbb{Z}$ with $k=k_1+k_2+k_3\geq 2$ there exists a $\mu$-stable instanton bundle $E$ with $c_2(E)=k_1e_1+k_2e_2+k_3e_3$ on $X$ such that
\[
\Ext_X^1(E,E)=4k-3, \qquad \Ext_X^2(E,E)=\Ext_X^3(E,E)=0
\]
and such that $E$ is generically trivial on lines.

In particular, there exists inside $MI(k_1e_1+k_2e_2+k_3e_3)$ a generically smooth irreducible component of dimension $4k-3$.
\end{theorem}
\begin{proof}
We will divide the proof in two steps. In the first one we will construct a torsion free sheaf with increasing $c_2$. In the second step we deform it to a locally free sheaf.

{\bf Step 1:} Defining a sheaf $G$ with increased $c_2$.

Let us consider a charge $k$ instanton bundle $E$ on $X$ with $c_2(E)=k_1e_1+k_2e_2+k_3e_3$. Suppose $E_{|_{l_i}}=\sO_{l_i}^{ 2}$, with $l_i$ is a generic line of each family $e_i$ and $\Ext^2(E,E)=\Ext^3(E,E)=0$.

Let us consider the short exact sequence
\begin{equation}\label{ses1}
0\rightarrow G \rightarrow E \rightarrow \sO_{l}\rightarrow 0.
\end{equation}
where $l$ is a general line in the family $e_1$. $G$ is a torsion free sheaf which is not locally free. Using the resolution of $\sO_{l}$:
\begin{equation}\label{risoluzione}
0\rightarrow \sO_X(-h_2-h_3)\rightarrow \sO_X(-h_2) \oplus \sO_X(-h_3) \rightarrow \sO_X \rightarrow \sO_{l} \rightarrow 0
\end{equation}
we obtain $c_1(\sO_{l})=0$ and $c_2(\sO_{l})=-e_1$ so using the sequence \eqref{ses1} we have that $c_1(G)=0$, $c_2(G)=(k_1+1)e_1+k_2e_2+k_3e_3$ and $c_3(G)=0$.

Now, applying the functor $\Hom(E,-)$ to \eqref{ses1} we obtain $\Ext^2(E,G)=0$. In fact we have $\Ext^2(E,E)=0$ by hypothesis and $\Ext^1(E,\sO_l)=0$ by Serre's duality since $E_{|_{l}}=\sO_{l}^{ 2}$. Now apply the contravariant functor $\Hom(-,G)$ to \eqref{ses1}. We have the following sequence
\[
\Ext^2(E,G) \rightarrow \Ext^2(G,G) \rightarrow \Ext^3(\sO_l,G).
\]
Now we show that $\Ext^3(\sO_l,G)=0$ in order to obtain $\Ext^2(G,G)=0$. By Serre's duality we have $\Ext^3(\sO_l,G)=\Hom(G,\sO_l(-2h))$. Consider the spectral sequence
\[
E_2^{p,q}=H^p(X, {\mathcal{E}xt}^q(A,B)) \Rightarrow \Ext^{p+q}(A,B)
\]
 with $A,B \in \text{Coh}(X)$. Setting $A=G$ and $B=\sO_l(-2h)$ we obtain
\[
\Hom(G,\sO_l(-2h))=H^0(\mathcal{H}om(G,\sO_l(-2h))).
\]
Now applying the functor $\mathcal{H}om(-,\sO_l(-2h))$ to the sequence \eqref{ses1}, we obtain
\begin{equation}\label{ses3}
0\rightarrow \mathcal{H}om(\sO_l,\sO_l(-2h))\rightarrow \mathcal{H}om(E,\sO_l(-2h)) \rightarrow \mathcal{H}om(G,\sO_l(-2h))\rightarrow \mathcal{E}xt^1(\sO_l,\sO_l(-2h))\rightarrow 0.
\end{equation}
Now $\mathcal{H}om(\sO_l,\sO_l(-2h))\cong \sO_l(-2h)$, $\mathcal{H}om(E,\sO_l(-2h))\cong \mathcal{H}om(\sO_X,\sO_l) \otimes E_{|_l}^\vee(-2h)\cong\sO_l^2(-2h)$ and $\mathcal{E}xt^1(\sO_l,\sO_l(-2h))\cong N_l(-2h)=\sO_l^2(-2h)$. If we split \eqref{ses3} in two short exact sequences we obtain
\[
0 \rightarrow \sO_l(-2h) \rightarrow \mathcal{H}om(G,\sO_l(-2h)) \rightarrow \sO_l^2(-2h) \rightarrow 0.
\]
We deduce $\mathcal{H}om(G,\sO_l(-2h))\cong \sO_l^3(-2h)$, thus $H^0(\mathcal{H}om(G,\sO_l(-2h)))\cong H^0(\sO_l^3(-2h))= 0$. Finally we obtain $\Ext^3(\sO_l,G)\cong \Hom(G,\sO_l(-2h))=0$ from which it follows $\Ext^2(G,G)=0$. This implies that $M_X(2,0,c_2(G))$ is smooth in the point correspondent to $G$. Now we show that $\Ext^3(G,G)=0$. Applying the contravariant functor $\Hom(-,G)$ to \eqref{ses1} we get a surjection
\[
\Ext^3(E,G)\to \Ext^3(G,G)\to 0.
\]
If we apply $\Hom(E,-)$ to \eqref{ses1} we obtain $\Ext^3(E,G)=\Ext^2(E,\sO_l)$ which vanishes sice $E_{|_l}= \sO_l^2$. Thus $\Ext^3(G,G)=0$ and In particular we have that the dimension of the component of $M_X(2,0,c_2(G))$ containing $G$ has dimension equal to $\dim \Ext^1(G,G)=1-\chi(G,G)$. Applying $\Hom(G,-)$, $\Hom(-,E)$ and $\Hom(-,\sO_l)$ to \eqref{ses1} we obtain
\[
\chi(G,G)=\chi(E,E)-\chi(E,\sO_l)-\chi(\sO_l,E)+\chi(\sO_l,\sO_l).
\]
By inductive hypothesis $\chi(E,E)=4-4k$. We compute the remaining terms in the equation. Applying $\Hom(-,E)$, $\Hom(E,-)$ and $\Hom(\sO_l,-)$ to \eqref{risoluzione}, a Riemann-Roch computation yields $\chi(E,\sO_l)=\chi(\sO_l,E)=2$ and $\chi(\sO_l,\sO_l)=0$, thus
\[
\dim \Ext^1(G,G)=1-\chi(G,G)=4k+1.
\]

Furthermore tensor \eqref{ses1} by $\sO_{m_i}(-h)$ where $m_i$ is a generic line from the family $e_i$. Since $m_i$ and $l$ are disjoint for each $i$,  tensoring by $\sO_{m_i}(-h)$ leaves the sequence exact. Using the fact that $E_{|_{m_i}}=\sO_{m_i}^{\oplus 2}$, we obtain $G_{|_{m_i}}=\sO_{m_i}^{\oplus 2}$ and in particular $H^0(G\otimes \sO_{m_i}(-h))=0$ for each $i$.

{\bf Step 2}: Deforming $G$ to a locally free sheaf $F$.

Now we take a deformation of $G$ in $M_X(2,0,c_2(G))$ and let us call it $F$. For semicontinuity $F$ satisfies
\[
H^0(X, F\otimes \sO_{l}(-h))=0 \ \text{and} \ H^1(X,F(-h))=0
\]
Our goal is to show that $F$ is locally free. Let us take $E'$ and $l'$ two deformation in a neighborhood of $E$ and $l$ respectively. The strategy is to show that if $F$ is not locally free, then he would fit into a sequence
\[
0 \rightarrow F \rightarrow E' \rightarrow \sO_{l'} \rightarrow 0.
\]
But such $F$'s are parametrized by a family of dimension $4k$: indeed we have a $(4k-3)$-dimensional family for the choice of $E'$, 2 for the choice of a line in the first family and we have 1 for $\p =\mathbb{P}(H^0(l',E_{|_{l'}}))$, since $E_{l'}\cong \mathcal{O}_{l'}^2$. But we showed that $G$, and hence $F$, moves over a $(4k+1)$-dimensional component in $M_X(2,0,c_2(G))$, so $F$ must be locally free.

Given such $F$ let us consider the natural short exact sequence
\begin{equation}\label{biduale}
0\rightarrow F \rightarrow F^{\vee \vee} \rightarrow T \rightarrow 0.
\end{equation}
Let us denote by $Y$ the support of $T$. Since we supposed $F$ not locally free, we have that $Y\neq \emptyset$. Furthermore $T$ is supported in codimension at least two. We say that $Y$ has pure dimension one.

In fact twisting \eqref{biduale} by $\sO_X(-h)$ we observe that if $H^0(X,F^{\vee \vee}(-h))\neq 0$ then a nonzero global section of $F^{\vee \vee}$ will induce via pull-back a subsheaf $K$ of $F$ with $c_1(K)=h$, which is not possible since $F$ is stable. So we have $H^0(X,F^{\vee \vee}(-h))\cong H^1(X,F(-h))=0$ which implies $H^0(X,T(-h))=0$. In particular $Y$ has no embedded points, i.e. is pure of dimension one. We want to show that $Y$ is actually a line.

Let $H$ be a general hyperplane section which does not intersect the points where $F^{\vee\vee}$ is not locally free. Tensor \eqref{ses1} by $\sO_H$. Since $H$ is general the sequence remains exact and $\sO_{l\cap H}$ is supported at one point, which represent the point where $G_H$ fails to be reflexive (in this case also locally free). $F$ is a deformation of $G$ and because of the choice of $H$, restricting \eqref{biduale} to $H$ does not
affect the exactness of the short exact sequence. Moreover $T_H$ is supported on points where $F_H$ is not reflexive. Since being reflexive is an open condition, by semicontinuity $T_H$ is supported at most at one point. But $Y$ cannot be empty and is purely one dimensional, thus $Y\cap H$ consists of one point and $Y$ must be a line $L$. Furthermore by semicontinuity $T$ is of generic rank one and we have $c_2(T)h=-1$ (see \cite[Example 15.3.1]{fulton}).

Now we prove that $F^{\vee\vee}$ is locally free. Twist \eqref{biduale} by $\sO_X(th)$ with $t<<0$. Considering the long exact sequence induced in cohomology we have $h^1(X,T(t))\leq h^2(X,F(t))$ because $h^1(X,F^{\vee \vee})=0$ by Serre's vanishing. Observe that $c=c_3(F^{\vee \vee})$ and $c_2(T)$ are invariant for twists.

By computing the Chern classes using \eqref{biduale} we have $c_3(T)=c$ and $c_3(T(th))=c-2thc_2(T)$. For $t<<0$ we have
\[
h^1(T(th))=-\chi(T(th))=(t+1)hc_2(T)-\frac{c}{2}.
\]
By semicontinuity we have $h^2(F(th))\leq h^2(G(th))$, but using \eqref{ses1} and Hirzebruch-Riemann-Roch formula we obtain $h^2(G(th))=h^1(\sO_{l_1}(t))=-(t+1)$ for $t<<0$. Now we have
\[
(t+1)hc_2(T)-\frac{c}{2}=h^1(T(th))\leq h^2(X,F(t))\leq -(t+1)
\]
so that
\begin{equation}\label{diseq1}
hc_2(T)\geq -1+\frac{c}{2(t+1)},
\end{equation}
which holds for all $t<<0$. Now using \eqref{diseq1} and substituting $hc_2(T)=-1$ we get $c\leq0$. Since $F^{\vee \vee}$ is reflexive, $c\geq 0$ so we obtain $c_3(T)=c=0$.

Now it remains to show that $F^{\vee \vee}$ is a deformation of $E$. The first step is to show that $L$ is a deformation of the line $l$. In order to do so we compute the class of $L$ in $A^2(X)$, which is represented by $c_2(T)=a_1e_1+a_2e_2+a_3e_3$. Consider a divisor $D=\beta_1 h_1+\beta_2 h_2+\beta_3 h_3$, by \eqref{RR} and $c=0$ we have
\[
h^1(L,T(D))=(D+2)c_2(T).
\]
Suppose $\beta_i <<0$ for all $i$. Then
\begin{equation}\label{equa1}
a_1(\beta_1+1)+a_2(\beta_2+1)+a_3(\beta_3+1)=h^1(L,T(D))=h^2(X,F(D))\leq h^2(X,G(D))
\end{equation}
where the last inequality is by semicontinuity. Furthermore $\beta_i <<0$ implies that $h^1(X,E(D))=h^2(X,E(D))=0$ and thus
\begin{equation}\label{equa2}
h^2(X,G(D))=h^1(l,\sO_l(D))=-1-\beta_1.
\end{equation}
We showed that $a_1+a_2+a_3=c_2(T)h=-1$ and combining this with \eqref{equa1} and \eqref{equa2} we obtain
\[
a_2(\beta_2-\beta_1)+a_3(\beta_3-\beta_1)\leq 0
\]
for all $\beta_i<<0$, thus we must have $a_2=a_3=0$ and $a_1=-1$, i.e. $L$ lives in a neighborhood of $l$. Since $c=0$ we have that $F^{\vee \vee}$ is locally free and we computed $c_2(T)=-e_1$, so we get $c_2(F^{\vee \vee})=k_1e_1+k_2e_2+k_3e_3$, which implies that $F^{\vee \vee}$ has the same Chern classes as $E$. Therefore, $F^{\vee \vee}$ is a flat deformation of $E$ and also semistable, so $F^{\vee \vee}$ lies in a neighborhood of $E$ in $M_X(2,0,c_2(E))$. Observe that, by semicontinuity, $F$ has trivial splitting type on the generic line of each family. To summarize, we showed that if $F$ is not locally free it fits into a sequence
 \[
0 \rightarrow F \rightarrow E' \rightarrow \sO_{l'} \rightarrow 0
\]
with $E'$ and $l'$ flat deformation of $E$ and $l$. But we observed that this is not possible, thus $F$ must be locally free.

\end{proof}

\section{Jumping lines}
In this section we describe the locus of jumping lines inside the Hilbert scheme of lines in $X$. Let us recall the definition of a jumping line:
\begin{definition}
Let $E$ be a rank two vector bundle on $X$ with $c_1(E)=0$. A jumping line for $E$ is a line $L$ such that $H^0(E_L(-r))=0$ for some $r>0$. The largest such integer is called the order of the jumping line $L$.
\end{definition}
Let us consider a line in the first family $e_1=h_2h_3$. Then we have the following resolution
\begin{equation}\label{risoluzioneL}
0\rightarrow \sO_X(-h_2-h_3)\rightarrow \sO_X(-h_2)\oplus \sO_X(-h_3)\rightarrow \sO_X \rightarrow \sO_L \rightarrow 0.
\end{equation}
Let $\mathcal{H}$ be the Hilbert scheme of lines of the family $h_2h_3$. In particular we have $\mathcal{H}=\p \times \p$, and we will denote by $l$ and $m$ the generators of $\Pic (\sH)$. Writing the sequence \eqref{risoluzioneL} with respect to global section of $\sO_X(-h_2)\oplus \sO_X(-h_3)$ we get the description of the universal line $\mathcal{L}\subset X\times \mathcal{H}$
\begin{equation}\label{universal}
0\rightarrow \sO_X(-h_2-h_3)\boxtimes \sO_\sH(-1,-1) \rightarrow
\begin{matrix}
\sO_X(-h_2)\boxtimes \sO_\sH(-1,0)\\
\oplus\\
\sO_X(-h_3)\boxtimes \sO_\sH(0,-1)
\end{matrix}
\rightarrow \sO_{X\times \sH} \rightarrow \sO_\sL \rightarrow 0.
\end{equation}
Let us denote by $\sD_E^1$ the locus of jumping lines (from the first family) of an instanton bundle $E$, and by $i$ its embedding in $\sH$. Let us consider the following diagram
\begin{equation}\label{diagramma}
\begin{tikzcd}[column sep=small]
& \sL \subset X\times \sH \arrow["q"',dl] \arrow[dr,"p"] & \\
  X &                         & \sH
\end{tikzcd}
\end{equation}
where $q$ and $p$ are the projection to the first and second factor respectively.
\begin{lemma}\label{support}
$\sD_E^1$ is the support of the sheaf $R^1p_\ast (q^\ast(E(-h_1))\boxtimes \sO_\sL)$.
\end{lemma}
\begin{proof}
See \cite[p.~108]{OSS} for a proof for $\pp^n$. Since the argument is local, it can be generalized to our case.
\end{proof}
We recall two classical result that we need in order to describe the locus of jumping lines.
\begin{theorem}[Grauert]\cite[Corollary 12.9]{hart}\label{grauert}
Let $f:X\rightarrow Y$ be a projective morphism of noetherian schemes with $Y$ integral, and let $F$ be a coherent sheaf on $X$, flat over $Y$. If for some $i$ the function $h^i(Y,F)$ is constant on $Y$, then $R^if_\ast (F)$ is locally free on $Y$, and for every $y$ the natural map
\begin{equation}\label{mappagrauert}
R^if_\ast(F)\otimes k(y)\rightarrow H^i(X_y,F_y)
\end{equation}
is an isomorphism.
\end{theorem}
\begin{theorem}\cite[Theorem 5.3, Appendix A]{hart}
Let $f:X\rightarrow Y$ be a smooth projective morphism of nonsingular quasi projective varieties. Then for any $x\in K(X)$ we have
\begin{equation}\label{GRR}
ch(f_!(x))=f_\ast(ch(x).td(T_f))
\end{equation}
in $A(Y)\otimes \mathbb{Q}$, where $T_f$ is the relative tangent sheaf of $f$.
\end{theorem}
Now we are ready to state the following
\begin{prop}
Let $E$ be a generic instanton on $X$ with $c_2=k_1e_1+k_2e_2+k_3e_3$. Then $\sD_E^1$ is a divisor given by $\sD_E^1=k_3l+k_2m$ equipped with a sheaf $G$ fitting into
\begin{equation}\label{risoluzionejumping}
0\rightarrow \sO_\sH^{k_3}(-1,0)\oplus \sO_\sH^{k_2} (0,-1)\rightarrow \sO_\sH^{k_2+k_3}\rightarrow i_\ast G\rightarrow 0.
\end{equation}
\end{prop}
\begin{proof}
By Lemma \ref{support} a line $L$ is jumping for $E$ if and only if the point of $\sH$ corresponding to $L$ lies in the support of $R^1p_\ast (q^\ast(E(-h_1))\boxtimes \sO_\sL)$.

Let us consider the Fourier-Mukai functor
\[
\Phi_{\sL}: D^b(X)\rightarrow D^b(\sH)
\]
with kernel the structure sheaf of $\sL$. We need to compute the transform of the bundles appearing in the monad \eqref{monade} tensorized by $\sO_X(-h_1)$.
\begin{itemize}
\item $\Phi_{\sL}(\sO_X(-2h_1-h_2))$.

By \eqref{universal} tensored by $\sO_X(-2h_1-h_2) \boxtimes \sO_\sH$, since the only non zero cohomology on $X$ is $h^2(\sO_X(-2h_1-2h_2))=1$ we get $R^i p_\ast(q^\ast(\sO_X(-2h_1-h_2))\boxtimes \sO_\sL)=0$ for $i\neq 1$. Using the projection formula we obtain
\[
R^1 p_\ast(q^\ast(\sO_X(-2h_1-h_2))\boxtimes \sO_\sL)\cong R^2 p_\ast(q^\ast(\sO_X(-2h_1-2h_2)))\boxtimes \sO_\sH(-1,0).
\]
Observe that by Theorem \ref{grauert} we have that $R^2 p_\ast(q^\ast(\sO_X(-2h_1-2h_2)))$ is a rank one vector bundle on $\sH$. Using \eqref{GRR} it follows trivially that \[ c_1(R^2 p_\ast q^\ast(\sO_X(-2h_1-2h_2)))=0. \] In fact consider the diagram \eqref{diagramma}. Since $X$ is a threefold and $\sH$ is a surface, we have that after being pulled-back on $X\times \sH$ and push-forwarded to $\sH$ all the cycles on $X$ became either zero or points. So we obtain

\[
R^1 p_\ast (q^\ast(\sO_X(-2h_1-h_2))\boxtimes \sO_\sL)\cong \sO_\sH(-1,0).
\]
We continue with the other terms of the monad \eqref{monade}. The computations are completely analogous.

\item $\Phi_{\sL}(\sO_X(-2h_1-h_3))$.

By \eqref{universal} tensored by $\sO_X(-2h_1-h_3) \boxtimes \sO_\sH$, since the only non zero cohomology on $X$ is $h^2(\sO_X(-2h_1-2h_3))=1$ we get $R^i p_\ast(q^\ast(\sO_X(-2h_1-h_3))\boxtimes \sO_\sL)=0$ for $i\neq 1$ and
\[
R^1 p_\ast(q^\ast(\sO_X(-2h_1-h_3))\boxtimes \sO_\sL)\cong \sO_\sH(0,-1).
\]

\item $\Phi_{\sL}(\sO_X(-h_1-h_2-h_3))$.

By \eqref{universal} tensored by $\sO_X(-h_1-h_2-h_3) \boxtimes \sO_\sH$, since the cohomology on $X$ is all zero we get $R^i p_\ast(q^\ast(\sO_X(-h_1-h_2-h_3))\boxtimes \sO_\sL)=0$ for all $i$.

\item $\Phi_{\sL}(\sO_X(-h_1-h_2))$.

By \eqref{universal} tensored by $\sO_X(-h_1-h_2) \boxtimes \sO_\sH$, since the cohomology on $X$ is all zero we get $R^i p_\ast(q^\ast(\sO_X(-h_1-h_2))\boxtimes \sO_\sL)=0$ for all $i$.

\item $\Phi_{\sL}(\sO_X(-h_1-h_3))$.

By \eqref{universal} tensored by $\sO_X(-h_1-h_3) \boxtimes \sO_\sH$, since the cohomology on $X$ is all zero we get $R^i p_\ast(q^\ast(\sO_X(-h_1-h_3))\boxtimes \sO_\sL)=0$ for all $i$.

\item $\Phi_{\sL}(\sO_X(-2h_1))$.

By \eqref{universal} tensored by $\sO_X(-2h_1) \boxtimes \sO_\sH$, since the only non zero cohomology on $X$ is $h^2(\sO_X(-2h_1))=1$ we get $R^i p_\ast(q^\ast(\sO_X(-2h_1))\boxtimes \sO_\sL)=0$ for $i\neq 1$ and
\[
R^1 p_\ast(q^\ast(\sO_X(-2h_1))\boxtimes \sO_\sL)\cong \sO_\sH.
\]

\item $\Phi_{\sL}(\sO_X(-h_1))$.

By \eqref{universal} tensored by $\sO_X(-h_1) \boxtimes \sO_\sH$, since the cohomology on $X$ is all zero we get $R^i p_\ast(q^\ast(\sO_X(-h_1))\boxtimes \sO_\sL)=0$ for all $i$.
\end{itemize}
Now we apply the $\Phi_\sL$ to the monad \eqref{monade}. First we apply $\Phi_\sL$ to the sequence
\[
0\rightarrow K \rightarrow
 \begin{matrix}
\sO_X^{k_2+k_3}(-h_1) \\
\oplus \\
\sO_X^{k_1+k_3}(-h_2) \\
\oplus \\
\sO_X^{k_1+k_2}(-h_3)
\end{matrix}
\rightarrow \sO_X^{k-2}\rightarrow 0
\]
we get $R^i p_\ast q^\ast(K\otimes \sO_X(-h_1))=0$ for $i\neq 1$ and
\[
R^1 p_\ast (q^\ast(K\otimes \sO_X(-h_1))\boxtimes \sO_\sL)\cong \sO_\sH^{k_2+k_3}.
\]
From
\[
0\rightarrow
\begin{matrix}
\sO_X^{k_3}(-h_1-h_2) \\
\oplus \\
\sO_X^{k_2}(-h_1-h_3) \\
\oplus \\
\sO_X^{k_1}(-h_2-h_3)
\end{matrix}
\rightarrow K \rightarrow E \rightarrow 0
\]
we get
\[
0\rightarrow R^0p_\ast (q^\ast (E \otimes \sO_X(-h_1))\boxtimes \sO_\sL)\rightarrow \begin{matrix} \sO_\sH^{k_3}(-1,0)  \\ \oplus \\ \sO_\sH^{k_2}(0,-1) \end{matrix} \overset{\gamma}{\rightarrow} \sO_\sH^{k_2+k_3} \to R^1 p_\ast (q^\ast (E \otimes \sO_X(-h_1))\to0 .
\]
so $\gamma$ is a $(k_2+k_3)\times (k_2+k_3)$ matrix made by two blocks. The first one is a $(k_2+k_3)\times (k_3)$ linear matrix in the first variables of $\sH$ and the second one a $(k_2+k_3)\times (k_2)$ linear matrix in the second variables of $\sH$. We observe that $\Ker (\gamma)$ is zero since is a torsion free sheaf which is zero outside $\sD_E^1$, and $\Coker (\gamma)\cong R^1 p_\ast (q^\ast (E \otimes \sO_X(-h_1))\boxtimes \sO_\sL)$ is an extension to $\sH$ of a rank 1 sheaf on $\sD_E^1$ denoted by $G$. That is a divisor $k_3l+k_2m$ given by the vanishing of the determinant of $\gamma$.

\end{proof}
\begin{remark}
The Hilbert space of lines on $X$ is made of three disjoint connected component, each of which is isomorphic to the quadric surface $\p \times \p$ (see \cite[Proposition 4.1]{p1p1p1}). So we can repeat this exact same reasoning to the lines of the family $e_2$ and $e_3$, i.e. permuting the indices $(1,2,3)$, we can describe the locus $\mathcal{D}_E^i$ as a divisor of type $k_jl+k_hm$ with $i\neq j \neq h \neq i$.

In a completely analogous way, it is possible to use the monad \eqref{monade2} to study the locus of jumping lines, obtaining the same result.
\end{remark}

\vspace{\baselineskip}
\bibliographystyle{plain}

\end{document}